\documentclass[11pt]{amsart}
\usepackage{amsmath,amssymb,amscd,amsthm,amsxtra}
\usepackage{latexsym}

\usepackage[dvips]{graphicx}        

\newtheorem{theorem}{Theorem}
\newtheorem{lemma}{Lemma}

\newcommand{\e}{\varepsilon}

\newcommand{\La}{\Lambda}

\newcommand{\p}{\partial}

\newcommand{\wh}{\widehat}

\newcommand{\ti}{\tilde}
\newcommand{\Id}{{\bf 1}}
\newcommand{\bQ}{{\bf Q}}
\newcommand{\bA}{{\bf A}}
\newcommand{\bB}{{\bf B}}
\newcommand{\bC}{{\bf C}}
\newcommand{\bD}{{\bf D}}
\newcommand{\bT}{{\bf T}}

\begin{document}

\title
[A bilinear oscillatory integral along parabolas ]
{A bilinear oscillatory integral along parabolas}

\author{Dashan Fan}
\address{
Dashan Fan\\
Department of Mathematics\\
University of Wisconsin-Milwaukee \& Huazhong Normal University\\
Milwaukee, WI, USA}
\email{fan@uwm.edu}

\author{Xiaochun Li}

\address{
Xiaochun Li\\
Department of Mathematics\\
University of Illinois at Urbana-Champaign\\
Urbana, IL, 61801, USA}

\email{xcli@math.uiuc.edu}

\date{\today}
\thanks{ The first author was partially supported by NSF grant of China.
The second author was supported by NSF grant DMS-0456976}


\begin{abstract}
We establish an $L^\infty\times L^2\rightarrow L^2$ norm estimate for a bilinear oscillatory integral operator along parabolas
incorporating oscillatory factors $e^{i|t|^{-\beta}}$.
\end{abstract}

\maketitle

\section{Introduction}

It is well-known that the Hilbert transform along curves:
$$ H_{\nu }f(x)=p.v.\int_{-1}^{1}f(x-\nu (t))\frac{dt}{t} $$
is bounded on $L^{p}({\mathbb {R}}^{n})$ for $1<p<\infty $, where $\nu (t)$ is an
appropriate curve in ${\mathbb R}^{n}$. 
Among various curves, one simple model case is the parabola $(t,t^{2})$ in the two
dimensional plane. This work was initiated by Fabes and Riviere \cite{FR} in
order to study the regularity of parabolic differential equations.
A nice survey \cite{SW} on this type of operators was written by Stein and Wainger. 
A lot of work on the Hilbert transform along curves 
had been done in the last thirty years by many people. Readers can find some of them in
\cite{Christ1, Christ2, DRubio, NVWW}.
The general results were established in \cite{CNSW} for the singular Radon transforms and 
their maximal analogues over smooth submanifolds of ${\mathbb{R}}^n$ with some curvature 
conditions.\\

The cancellation condition of $p.v.\frac{1}{t}$ plays an important role 
for obtaining $L^p$ boundedness of the Hilbert transform.  However,
this condition is not necessary if there is an oscillatory factor $e^{i
|t|^{-\beta }}$($\beta >0$) in the kernel (see \cite{W, FS, H}). Due to the high oscillation
of the factor $e^{i|t|^{-\beta}}$, $L^p$ estimates can be obtained for 
 corresponding operators with the kernel $ e^{i|t|^{-\beta}}/|t|$.
In \cite{Z}, Zielinski studied the following  oscillatory integral
$$
T_{\alpha ,\beta, \nu}(f)(x)=\int_{0}^{1}f(x-\nu (t))\,e^{it^{-\beta }}%
\frac{dt}{t^{1+\alpha }},\,\,\,\,\,{\rm with}\,\, \nu (t)=(t,t^{2}).
$$
He proved that $T_{\alpha ,\beta ,\nu }$ is bounded on $L^{2}(R^{2})$ if
and only if $\beta \geq 3\alpha $. Chandrana \cite{C} obtained
 $L^{2}({\mathbb R}^{2})$ boundedness of $T_{\alpha ,\beta ,\nu }$ for 
curves $\nu (t)=(t,t^{k}), k>1$. Recently, this result was
extended to high dimensions for curves $\nu
(t)=(t^{k_{1}},t^{k_{2}},..., t^{k_{n}})$ with $0<k_{1}<k_{2}<...<k_{n}$
(see \cite{CFWZ}).\\

In this article, we are interested in the
bilinear oscillatory integral along a parabola,

\begin{equation}\label{defofTbeta}
T_{\beta }(f, g)(x)=\int_{-1}^{1}f(x-t)\;g(x-t^{2})\,e^{i |t| ^{-\beta }}
\frac{dt}{|t| } \,, \,\,{\rm where}\,\, \beta>0\,.
\end{equation}

The main theorem that we prove is the following. 

\begin{theorem}\label{thm1}
If $\beta >1$, then the operator $T_\beta$ is bounded from $L^{\infty}\times L^2 $ to $L^2$, that is, 
\begin{equation}\label{022}
\big\|T_\beta(f, g)\big\|_2\leq C\|f\|_\infty\|g\|_2\,
\end{equation}  
for all $f\in L^\infty$ and $g\in L^2$. 
\end{theorem}

We can view this result as a bilinear version of the operator $T_{\alpha
,\beta, \nu } $ in the case $\alpha$ = 0. 
It was observed that all proofs of
the $L^{2}$ boundedness mentioned in \cite{C, CFWZ, Z} were mainly based on
Plancherel's Theorem. 
However, our proof is much more difficult than those in the linear case.
Moreover, our method can also be used to handle 
the kernel $1/|t|^{1+\alpha}$ with stronger singularity.  With a little more technical 
modification, this method also works for the operator along a polynomial curve if
one replaces $t^2$ by a polynomial $P(t)$.   For simplicity, we only concentrate on the $t^2$ case here.
We do not know yet what is the best lower bound for $\beta$. A natural guess
would be $0$, however, the method in this paper does not give any 
lower bound better than $1$.
By the time-frequency analysis, it is possible to get
$L^p\times L^q\rightarrow L^r$ estimates for $T_\beta$ 
for all $p, q>1$ and $1/r=1/p+1/q$. 
The more general curve cases and $L^r$ estimates will appear 
in subsequent papers.  A more interesting problem 
is the following, which yields $L^r$ estimates immediately. \\

\noindent
{\bf Question 1.}
Let $\rho$ be a standard bump function supported on the interval $[1/8, 1/2]$ or $[-1/2, -1/8]$.  And let $\beta > 0$,
$j\geq 1$, and 
$T_{j,\beta}(f,g)$ be defined by
$$
T_{j,\beta}(f,g)(x)=\int f(x-t)g(x-t^2)e^{i|t|^{-\beta}}2^j\rho(2^jt)dt\,.
$$
Are there positive constants 
$C$ and $\e$ independent of 
$f, g$ and $j$ such that 
\begin{equation}\label{estTj}
\|T_{j,\beta}(f,g)\|_{r}\leq C2^{-\e j}\|f\|_p\|g\|_q\,
\end{equation} 
holds for some $p>1,q>1$ and $ 1/p+1/q=1/r$?\\

When $\beta >0$, $e^{i|t|^{-\beta}}$ is a highly oscillatory factor as $t\rightarrow 0$, 
Thus it is natural to ask this kind of question, and seems very possible to 
get an affirmative answer to this question if $\beta$ is large enough,
say $\beta> 6$. There are two possible ways to solve this question. 
One of them is to ask whether there exists a positive number $\e$ such that
$$
\bigg|\int\int f_1(x) f_2(y) e^{\lambda u(x, y)}\chi(x,y) dx dy\big|\leq C
 (1+|\lambda|)^{-\e} \|f_1\|_2\|f_2\|_2
$$
holds for all $f_1, f_2\in L^2$ if $ |L(u) |\geq C$, where
$L$ is a suitable differential operator and $\chi$ is a suitable 
bump function on a bounded set. However this seems to be a quite 
challenging way.  A lot of work 
had been done for $L=\p^m\p^n/\p x^m\p y^n$. For example, some of this
type of work can be found in {\cite{CCW}} and {\cite{PSt1}}. A more promising way is to consider 
the phase function by a delicate analysis on the stationary phase. 
The main difficulty seems to be the stability of the critical points of 
the phase function $ a\xi t+b\eta t^2+f(t)$ for some $a, b\in \mathbb R$ and $C^\infty$ function
$f$, when the second order derivative of the phase function can be very small. 
By stability, we mean that some properties of the critical points can not
be destroyed when there is a perturbation of the variables $(\xi, \eta)$.
The desired stability can be obtained when the second order derivative of the phase
function is large, which is one of the crucial points in this paper.  
A further investigation on the stability of the critical points will be carried out.\\

Following the work of Lacey and Thiele, \cite{LT}, the field of multi-linear 
operators has been actively developed, to the point that some of the most 
interesting open questions have a strong connection to some kind of non-abelian 
analysis.  For instance, the tri-linear Hilbert transform 
\begin{equation*}
\int f_1 (x+y) f_2 (x+2y) f _3 (x+3y) \frac {dy} y 
\end{equation*}
has a hidden quadratic modulation symmetry which must be accounted for in any 
proposed method of analysis.  This non-abelian character is explicit in the 
work of B.~Kra and B.~Host \cite{HK} who characterize the 
characteristic factor of the corresponding ergodic averages 
\begin{equation*}
N ^{-1} \sum _{n=1} ^{N} f_1 (T ^{n}) f_2 (T ^{2n}) f_3 (T ^{3n}) 
\longrightarrow \prod _{j=1} ^{3} \mathbb E (f_j \mid \mathcal N)
\end{equation*}
Here, $ (X, \mathcal A, \mu ,T )$ is a measure preserving system, $ \mathcal N\subset
\mathcal A$ is the sigma-field which describes the characteristic factor.  In this case, 
it arises from certain $ 2$-step nilpotent groups.  The limit above is in the sense 
of $ L ^{2}$-norm convergence, and holds for all bounded $ f_1,f_2,f_3$. 

The ergodic analog of  the bilinear Hilbert transform along a parabola is 
\begin{equation*}
N ^{-1} \sum _{n=1} ^{N} f_1 (T ^{n}) f_2 (T ^{n ^2 })
\longrightarrow \prod _{j=1} ^{2} \mathbb E (f_j \mid \mathcal K _{\textup{profinite}})
\end{equation*}
where $ \mathcal K _{\textup{profinite}}\subset \mathcal A$ is the profinite factor, 
a subgroup of the maximal abelian factor of $ (X, \mathcal A, \mu ,T)$.  This last point 
suggests that Fourier analysis might be able to successfully analysize the bilinear 
Hilbert transform along parabola.  However, the proof  of the characteristic factor result  
above, due to Furstenberg \cite{Fur}, utilizes the characteristic factor for the 
three-term result.  (We are indebted to M. Lacey for bringing Furstenberg's theorems to our attention.)
This suggests that  the bilinear Hilbert transform along parabolas seems to be a
result at the very edge of what might be understood by Fourier analytic techniques.   
Perhaps  time-frequency analysis should be combined with 
estimates for the tri-linear oscillatory integrals (Lemma \ref{osc4bj}) studied in this paper. \\

\noindent
{\bf Acknowledgement} The second author would like to thank 
his wife, Helen,  and his son, Justin, for being together through
the hard times in the past two years. And he is also very grateful to 
Michael Lacey for his constant support and encouragement. The authors would
like to express their gratitude to  the very skillful and generous referee
for his many valuable comments and suggestions, which resulted in 
a great improvement in presentation of this paper.

\section{A Reduction}\label{secRe}
\setcounter{equation}0

In this section, we first show that Theorem \ref{thm1} can be reduced to Theorem \ref{thm3}. 

\begin{theorem}\label{thm3}
Let $\rho$ be a standard bump function supported on the interval $[1/8, 1/2]$ or $[-1/2, -1/8]$.  And let 
$\beta>1$, $j\geq 1$ and 
$T_{j, \beta}(f,g)$ be defined by
$$
T_{j, \beta}(f,g)(x)=\int f(x-t)g(x-t^2)e^{i|t|^{-\beta}}2^j\rho(2^jt)dt\,.
$$
Then 
\begin{equation}\label{estTj}
\big\|\sum_{j=1}^{\infty}T_{j, \beta}(f,g)\big\|_{2}\leq C\|f\|_\infty\|g\|_2\,
\end{equation} 
holds for all $f\in L^\infty$ and $g\in L^2$.
\end{theorem}

Recall that 
$\rho$ is a suitable standard bump function supported on the interval $[1/8, 1/2]$
or $[-1/2,-1/8]$. Let $\ti\rho(t)=\rho(t)+\rho(-t)$. 
For $|t|\leq 1$, we write 
$$
\frac{1}{|t|} =C \sum_{j=1}^{\infty}2^{j}\ti\rho(2^jt) + K_0(t),
$$
where $C$is a (unimportant) constant and $K_0(t)$ is a  bounded function supported on $1/4<|t|<1$. 
Then clearly Theorem \ref{thm1} is a consequence of 
Theorem \ref{thm3} and the following theorem.

\begin{theorem}\label{thm2}
Let $T(f,g)(x)= \int_{1/4<|t|<1}|f(x-t)g(x-t^2)|dt$. Then $T$ is bounded from 
$L^{p}\times L^q $ to $L^r$ for all $1<p, q \leq \infty$ and $ 1/p+1/q=1/r$. 
\end{theorem}

\begin{proof}
The only bad (singular) point in $1/4\leq |t|\leq 1$ is $t=1/2$. We will decompose 
$1/4\leq |t|\leq 1$ into a union of intervals such that the distance between $1/2$ 
and each interval in the union is comparable to the length of the interval. This is 
essentially the Whitney decomposition. Then we should show that there is a desired 
decay estimate for the corresponding integral over each interval in the previous 
decomposition. These decay estimates allow us to sum all intervals together.    \\

Indeed, 
we may without loss of generality restrict $x$, hence likewise the supports
of $f, g$, to fixed bounded intervals. This is possible because of the restriction
$|t|\leq 1$ in the integral. The trouble happens at a neighborhood of $t=1/2$ since 
the Jocobian $\frac{\partial (u,v)}{\partial (x,t)}=1-2t$ if $u=x-t$ and $v=x-t^2$.
We only prove the bounds for the integral operators with $1/2<|t|<1$ since another part 
$1/4<|t|<1/2$ can be handled similarly. 
Let $\psi$ be a standard bump function supported in $[-100, 100]$.
By changing variables, we only need to show that
\begin{equation}\label{estlem}
\int \bigg|\int_{0<|t|<1/2} f(x-t)g(x-t-t^2+1/4)dt \bigg|^r\psi(x)dx 
\leq C\|f\|_{p}^r\|g\|_q^r\,,
\end{equation} 
for $p>1,q>1$ and $r>1/2$ with $ 1/p+1/q=1/r$. 
Let $\varphi$ be a suitable standard bump function supported in $1/8<|t|<1/2$. It suffices to 
prove that there is a positive $\e$
\begin{equation}\label{estlem1}
\int \bigg|\int f(x-t)g(x-t-t^2)\varphi(2^jt)dt \bigg|^r \psi(x)dx \leq C2^{-\e j}\|f\|_{p}^r\|g\|_q^r\,,
\end{equation} 
for all $j\geq 1$, $p>1,q>1$ and $r>1/2$ with $ 1/p+1/q=1/r$, since (\ref{estlem}) follows 
by summing for all $j\geq 1$. Let $A_N =[-2^{-j-1}-100+N2^{-j}, 
-2^{-j-1}-100+(N+1)2^{-j}]$ for $N=0, \cdots, 200\cdot 2^{j}$. And 
let ${\bf 1}_{A_N}$ be the characteristic function of $A_N$.  
Notice that for a fixed $x\in [-100, 100]$, $x-t-t^2$ is in $A_{N-1}\cup A_{N}\cup {A_{N+1}}$ for some $N$
whenever $t$ is in the support of $\varphi(2^j\cdot)$.
Thus we can restrict $x$ in one of $A_N$'s  so that it suffices to show
that 
\begin{equation}\label{estlem2}
 \int \bigg|\int f_N(x-t)g_N(x-t-t^2)\varphi(2^jt)dt \bigg|^r \psi(x)dx 
\leq C2^{-\e j}\|f\|_{p}^r\|g\|_q^r\,
\end{equation} 
for all $j\geq 1$, $p>1,q>1$ and $r>1/2$ with $ 1/p+1/q=1/r$, where 
$f_N = f {\bf 1}_{A_N}$, $g_N=g{\bf 1}_{A_N}$  and $C$ is independent of $N$.
Let $T_N(f,g)(x) =\int f_N(x-t)g_N(x-t-t^2)\varphi(2^jt)dt $.
 By inserting absolute values throughout 
we get $T_N$ maps $L^p\times L^q$ to $L^r$ with a bound $C2^{-j}$ 
uniform in $N$, whenever $(1/p, 1/q, 1/r)$ belongs to the closed convex 
hull of the points $(1,0,1)$, $(0,1,1)$ and $(0,0,0)$.
Observe that by Cauchy-Schwarz inequality, 
\begin{equation}\label{estlem3}
 \int \big| T_N(f, g)(x)\big|^{1/2} \psi(x)dx
  \leq 2^{-j/2} \|T_{N}(f,g)\|_1^{1/2}
 \leq C\|f\|_1^{1/2}\|g\|_1^{1/2}\,.
\end{equation} 
Hence an interpolation yields a bound $C2^{-\e j}$ for 
all triples of reciprocal exponents within the convex hull of
$(1,1,2)$, $(1,0,1)$, $(0,1,1)$ and $(0,0,0)$. This finishes the proof
of Theorem \ref{thm2}.

\end{proof}

\section{A Decomposition}\label{para1}
\setcounter{equation}0

We begin the proof of our main Theorem by constructing 
an appropriate decomposition of the operator $ T _{j,\beta }$.  This is done 
by an analysis of the bilinear symbol associated with the operator. \\

A change of variables gives  
$$
 T_{j, \beta}(f, g)(x) = \int f(x-2^{-j}t)g(x-2^{-2j}t^2)e^{i2^{\beta j}/|t|^{\beta}} \rho(t) dt\,.
$$
Expressing $T_{j, \beta}$ in dual frequency variables, we have 
\begin{equation}\label{ft}
T_{j,\beta}(f,g)(x) = \int \int {\widehat{f}}(\xi){\widehat{g}}(\eta)
  e^{i(\xi+\eta)x} m_{j,\beta}(\xi,\eta)d\xi d\eta\,,
\end{equation}
where $m_{j, \beta}$ is the bilinear symbol of $T_{j, \beta}$, which equals to
$$
m_{j,\beta}(\xi,\eta) =\int \rho(t)e^{-i(2^{-j}\xi t + 2^{-2j}\eta t^2 -
 2^{\beta j} |t|^{-\beta})} dt\,.
$$

We introduce a resolution of the identity. 
Let $\Theta $ be a Schwarz function supported on $(-1,1)$ such 
that $\Theta(\xi)=1$ if $|\xi|\leq 1/2$. 
Set $\Phi$ to be a Schwartz function satisfying 
$$\wh\Phi(\xi)=\Theta(\xi/2)-\Theta(\xi)\,.$$
Then $\Phi$ is a Schwartz function such that $\wh\Phi$ is supported
on $\{\xi: 1/2 < |\xi| < 2\}$ and 
\begin{equation}\label{defofphi}
\sum_{m\in\mathbb Z}\wh\Phi\big(\frac{\xi}{2^m}\big)=1\,\, {\rm for}
\,\,{\rm all}\,\, \xi\in \mathbb R\backslash \{0\}\, ,
\end{equation}
and for any $m_0\in \mathbb Z$, 
\begin{equation}\label{defofphi0}
{\wh \Phi_{m_0}}(\xi) = \sum_{m=-\infty}^{m_0}
\wh\Phi{\big(\frac{\xi}{2^m}\big)} = \Theta\big(\frac{\xi}{2^{m_0+1}}\big)\,,
\end{equation}
which is a bump function supported on $(-2^{m_0+1}, 2^{m_0+1})$. 

We decompose the operator $T_{j,\beta}$ into 
$$
 T_{j, \beta} = \sum_{m, m'\in \mathbb Z} T_{m, m', j, \beta}\,, 
$$
where $T_{m, m', j, \beta}$ is defined by 
\begin{equation}\label{defofTmjb}
T_{m, m', j,\beta}(f,g)(x) = \int \!\int
 {\widehat{f}}(\xi){\widehat{g}}(\eta)
    e^{i(\xi+\eta)x} \wh\Phi\big( \frac{\xi}{2^{m+\beta j+j}}\big) 
\wh\Phi\big( \frac{\eta}{2^{m'+\beta j +2j }}\big)  m_{j,\beta}(\xi,\eta)d\xi d\eta\,,
\end{equation}
Let $b_\beta$ be a very large number depending on $\beta$. For $\beta >1$, we can choose  
$b_\beta = [ 100\beta^{100} ]$, where $[x]$ denotes the largest integer no more than $x$.  We then decompose $T_{j, \beta}$ into
\begin{align*}
T _{b,\beta ,\ell }&=\sum _{ (m,m')\in \Gamma _{\ell }} T _{m, m', j, \beta}, 
\qquad 1\le \ell \le 8 \,, 
\\ 
\Gamma_{1}&=\{ (m,m')\in \mathbb Z^2\mid {m\leq 10b_\beta, -b_\beta \leq m' \leq
b_\beta}\}\,, 
\\  
\Gamma_{2}&=\{ (m,m')\in \mathbb Z^2\mid {m> 10b_\beta, -b_\beta \leq m'\leq  b_\beta}\}\,, 
\\
\Gamma_{3}&=\{ (m,m')\in \mathbb Z^2\mid {m\leq -b_\beta,  m'< -b_\beta}\}\,, 
\\
\Gamma_{4}&=\{ (m,m')\in \mathbb Z^2\mid {-b_\beta< m < b_\beta, m'<-b_\beta}\}\,, 
\\
\Gamma_{5}&=\{ (m,m')\in \mathbb Z^2\mid {m\geq b_\beta, m'< -b_\beta}\}\,, 
\\
\Gamma_{6}&=\{ (m,m')\in \mathbb Z^2\mid {m\leq -b_\beta, m'> b_\beta}\}\,,
\\
\Gamma_{7}&=\{ (m,m')\in \mathbb Z^2\mid {-b_\beta < m < b_\beta, m'>b_\beta}\}\,, 
\\
\Gamma_{8}&=\{ (m,m')\in \mathbb Z^2\mid {m\geq b_\beta,  m' > b_\beta}\}\,.
\end{align*}

Let $\phi_{\xi, \eta}(t)=2^m\xi t +2^{m'}\eta t^2-|t|^{-\beta}$.
Define $\ti m(\xi, \eta)$ by 
$$
 \ti m(\xi, \eta) = \int \rho(t) e^{-i2^{\beta j}\phi_{\xi, \eta}(t)} dt\,.
$$
$\phi_{\xi, \eta}$ depends on $m, m'$ and $\ti m $ depends on $j$ but
we suppress the dependence for notational convenience. 
Heuristically, we decompose the operator according to the 
occurrence of the critical points of the phase function 
$\phi_{\xi, \eta}(t)=2^m\xi t +2^{m'}\eta t^2-|t|^{-\beta}$ 
and $\phi'_{\xi, \eta}$ for $\xi, \eta\in {\rm supp}\wh\Phi$. \\

In cases $T_{j, \beta, 2}, T_{j, \beta, 3}, 
T_{j, \beta, 5}, T_{j, \beta, 6}$, the phase function does not have any critical point, and in fact one 
can obtain a very rapid decay of $ O(2 ^{-M \beta j})$ for these cases (see Section \ref{product}). 
In the cases $ T_{j, \beta, 4}$, $ T_{j, \beta, 7}$ and $ T_{j, \beta, 8}$ a critical point 
of the phase function can occur, and therefore the methods of stationary phase 
must be brought to bear in these cases, exploiting in particular the oscillatory term.  
These terms require the most extensive analysis. 
The case of $ T_{j, \beta, 1}$ doesn't fall in the either of the preceding cases, but is 
straight forward to control, as it is can be viewed  as essentially a para-product operator (see Section \ref{para-product}).

\section{Sum of $T_{j, \beta, 1}$'s }\label{para-product}
\setcounter{equation}0

Observe that $T_{j, \beta, 1}(f,g)$ equals to
$$
\sum_{-b_\beta\leq m'\leq b_\beta}\int\!\!\int
\wh f(\xi)\wh g(\eta)e^{i(\xi+\eta)x}\Theta\big(\frac{\xi}
{2^{ 10b_\beta+1 +\beta j +j}}\big) \wh\Phi\big(\frac{\eta}{2^{m'+\beta j +2j}} \big)
 m_{j, \beta}(\xi, \eta) d\xi d\eta\,.
$$
If $j$ is large enough (larger than some constant depending on $\beta$),
then $ 2^{m'+\beta j + 2j - 3}\leq |\xi+\eta| \leq  2^{m'+\beta j + 2j + 3}$ whenever 
$\xi$, $\eta$ are in the supports of the respective dilates of $\Theta$ and $\wh\Phi$.
Let $\Phi_3$ be a Schwartz function such that 
$\wh\Phi_3 $ is supported in $(1/16, 9)\cup (-9, -1/16)$ such that 
$\wh\Phi_3(\xi)=1$ if $ 1/8\leq |\xi|\leq 8 $. Then for large $j$, we
have 
$$
\langle T_{j, \beta, 1}(f,g), h\rangle
= \sum_{-b_\beta\leq m'\leq b_\beta}\int\!\!\int
\wh f_{j}(\xi)\wh g_{j,m'}(\eta)\wh h_{j,m'}(\xi +\eta)
 m_{j, \beta}(\xi, \eta) d\xi d\eta\,,
$$
where $f_j$, $g_{j, m'}$ and $h_{j, m'}$ satisfy
$$
\wh f_{j}(\xi) = \wh f(\xi)\Theta\big(\frac{\xi}{2^{ 10b_\beta+1 +\beta j +j}}\big)\,,
$$
$$
\wh g_{j,m'}(\eta)=\wh g(\eta)\wh\Phi\big(\frac{\eta}{2^{m'+\beta j +2j}} \big)\,,
$$
$$
\wh h_{j,m'}(\xi)= \wh h(\xi)\wh\Phi_3\big(\frac{\xi}{2^{m'+\beta j +2j}} \big)\,.
$$
We can also write $\langle T_{j, \beta, 1}(f,g), h\rangle$ by
$$
 \sum_{-b_\beta\leq m'\leq b_\beta}
 \int\rho(t)\bigg( \int f_j(x-2^{-j}t) g_{j, m'}(x-2^{-2j}t^2) h_{j, m'}(x) dx \bigg)    e^{i2^{\beta j}|t|^{-\beta}}dt\,.
$$
Summing all $j$ and applying Cauchy-Schwarz inequality, we dominate
$\big| \big\langle \sum_jT_{j, \beta, 1}, h\big\rangle\big|$ by
$$
 \|f\|_\infty \sum_{-b_\beta\leq m'\leq b_\beta}
 \int\int|\rho(t)|
\bigg(\sum_j\big| g_{j, m'}(x-2^{-2j}t^2)\big|^2\bigg)^{1/2} 
 \bigg(\sum_j\big|h_{j, m'}(x)\big|^2\bigg)^{1/2}dx dt\,,
$$
which, by one more use of Cauchy-Schwarz inequality,
is clearly majorized  by
$$
 C\sum_{-b_\beta\leq m'\leq b_\beta}
\|f\|_\infty \bigg\|\bigg(\sum_j\big| g_{j, m'}\big|^2\bigg)^{1/2}\bigg\|_2 \bigg\|\bigg(\sum_j\big| h_{j, m'}\big|^2\bigg)^{1/2}\bigg\|_2 \,.
$$
Littlewood-Paley Theorem then yields
\begin{equation}\label{est022T1}
\bigg| \big\langle \sum_jT_{j, \beta, 1}(f,g), h\big\rangle\bigg|
\leq C_\beta \|f\|_\infty\|g\|_2\|h\|_2\,.
\end{equation}
Therefore we obtain  
\begin{equation}\label{1est022T1}
\bigg\| \sum_jT_{j, \beta, 1}(f,g)\bigg\|_2\leq C_\beta \|f\|_\infty\|g\|_2\,.
\end{equation}

\section{The Simplest Case}\label{product}
\setcounter{equation}0

In this section we deal with the
cases $T_{j, \beta, 2}, T_{j, \beta, 3}, T_{j, \beta, 5}, T_{j, \beta, 6}$.  

\begin{lemma}\label{lemsimple}
Let $j, \beta\geq 0$ and $\ell=2, 3, 5, 6$. For any positive integer $M$
there is a constant $C$ such that 
\begin{equation}\label{est137}
 \big\|T_{j, \beta, \ell}(f, g)\big\|_r\leq C_M 2^{-\beta Mj}\|f\|_p\|g\|_q
\end{equation}
holds for all $1<p, q\leq \infty$ and $1/r=1/p+1/q$.
\end{lemma}
\begin{proof}

First we prove the case $\ell=3$. From (\ref{defofphi0}), 
we see that 
\begin{equation}\label{T11}
 T_{j, \beta, 3}(f, g)(x)
 = \int \!\int{\widehat{f}}(\xi){\widehat{g}}(\eta)
    e^{i(\xi+\eta)x} \Theta\big( \frac{\xi}{2^{-b_\beta+1+\beta j+j}}\big) 
\Theta\big( \frac{\eta}{2^{-b_\beta+1+\beta j +2j }}\big)  m_{j,\beta}(\xi,\eta)d\xi d\eta\,,
\end{equation}
Let $\ti m_{3, j, \beta}$ be defined by 
$$
\ti m_{3, j, \beta}(\xi, \eta)=\int \rho(t)e^{-i 2^{\beta j} 
\phi_{3, j,\xi, \eta}(t) } dt
$$
where
$$
\phi_{3, j, \xi, \eta}(t) = 2^{-b_\beta+1}\xi t + 2^{-b_\beta+1}\eta t^2 - |t|^{-\beta} \,.
$$
And it is clear by the definition of $b_\beta$ that 
\begin{equation}\label{simple1est1}
 \big|\phi'_{3, j, \xi, \eta}(t)\big|\geq C_\beta \,.
\end{equation}

Let $\Theta_1$ be a Schwartz function supported on 
$|\xi|< 3/2 $ and ${\Theta_1}(\xi)=1$ if $|\xi|\leq 1$.
An integration by parts gives that
\begin{equation}\label{dest11}
\big| \p^{\alpha_1}_\xi\p^{\alpha_2}_\eta 
\big({\Theta_1}(\xi){\Theta_1}(\eta)\ti m_{3, j, \beta}(\xi, \eta)\big)\big|
\leq C_{M, \beta}2^{-\beta Mj}
    \big(1+|\xi|+|\eta|\big)^{-(\alpha_1+\alpha_2)}
\end{equation}
holds for all non-negative integers $\alpha_1, \alpha_2$ and $M$. 
Then we expand this function into its Fourier series to obtain 
\begin{equation}
 \big({\Theta_1}(\xi){\Theta_1}(\eta)\ti 
m_{3, j, \beta}(\xi, \eta)\big) = \sum_{n_1, n_2} 
C_{n_1, n_2}e^{2\pi i n_1\xi+2\pi i n_2\eta}\,,
\end{equation}
where the Fourier coefficients $C_{n_1, n_2}$'s satisfy
\begin{equation}\label{Fest11}
 |C_{n_1, n_2}|\leq C_{M, \beta}2^{-\beta Mj} (1+|n_1|)^{-M}(1+|n_2|)^{-M}
\end{equation} 
for all $M\geq 0$. Changing variables, we obtain 
$$
 {\Theta_1}\big( \frac{\xi}{2^{-b_\beta +1+\beta j+j}}\big) 
{\Theta_1}\big( \frac{\eta}{2^{-b_\beta+1+\beta j +2j }}\big) 
 m_{j,\beta}(\xi,\eta)
\!\!= \!\!\!\sum_{n_1, n_2} C_{n_1, n_2}e^{2\pi i n_12^{b_\beta-1-\beta j-j}\xi+2\pi i n_2
 2^{b_\beta-1-\beta j -2j }\eta}\,,
$$ 
since $m_{j, \beta}(\xi, \eta)=
\ti m_{3, j, \beta}(\xi/2^{-b_\beta+1+\beta j +j}, 
\eta/2^{-b_\beta+1+\beta j+2j})$. 
And then we can write $T_{j, \beta, 3}$ as a product, i.e., 
$$
 T_{j, \beta, 3}(f, g)(x)=\sum_{n_1, n_2}C_{n_1, n_2}f_{n_1, j}(x) 
g_{n_2, j}(x)\,,
$$
where 
$$
 {\wh f_{n_1, j}}(\xi) = \wh f(\xi) e^{2\pi in_1\xi/2^{-b_\beta+1+\beta j+j}}
  \Theta\big(\frac{\xi}{2^{-b_\beta+1+\beta j+j}} \big)\,
$$
$$
 {\wh g_{n_2, j}}(\eta) = \wh g(\eta) e^{2\pi in_2\eta/2^{-b_\beta+1+\beta j+2j}}
  \Theta\big(\frac{\eta}{2^{-b_\beta+1+\beta j+2j}} \big)\,
$$
Clearly (\ref{Fest11}) yields (\ref{est137}) since it is trivial to get 
the $L^r$ estimates for the product of two functions.  \\

We now turn to the proof for the case $\ell=6$. 
From (\ref{defofphi0}), we have  that 
\begin{equation}\label{T31}
 T_{j, \beta, 6}(f, g)(x)
 =\sum_{m'> b_\beta} \int \!\int{\widehat{f}}(\xi){\widehat{g}}(\eta)
    e^{i(\xi+\eta)x} \Theta\big( \frac{\xi}{2^{-b_\beta+1+\beta j+j}}\big) 
\wh\Phi\big( \frac{\eta}{2^{m'+\beta j +2j }}\big)  m_{j,\beta}(\xi,\eta)d\xi d\eta\,,
\end{equation}
Let $\ti m_{6, j, \beta}$ be defined by 
$$
\ti m_{6, j, \beta}(\xi, \eta)=\int \rho(t)e^{-i 2^{\beta j} 
\phi_{6, j,\xi, \eta}(t) } dt
$$
where
$$
\phi_{6, j, \xi, \eta}=2^{-b_\beta+1}\xi t + 2^{m'}\eta t^2 - |t|^{-\beta} \,.
$$
Then  the definition of $b_\beta$ and the fact $m'> b_\beta$
gives 
\begin{equation}\label{simple1est3}
 \big|\phi'_{6, j, \xi, \eta}(t)\big|\geq C_\beta 2^{m'}\,.
\end{equation} 
Let $\Phi_6$ be a function such that $\wh\Phi_6$ is  a Schwartz function supported on 
$1/4<|\xi|< 5/2 $ and $\wh{\Phi_6}(\xi)=1$ if $1/2\leq |\xi|\leq 2$.

By integration by parts we get that
\begin{equation}\label{dest31}
\big| \p^{\alpha_1}_\xi\p^{\alpha_2}_\eta 
\big({\Theta_1}(\xi){\wh{\Phi_6}}(\eta)\ti m_{6, j, \beta}(\xi, \eta)\big)\big|
\leq C_{M, \beta}2^{-\beta Mj}2^{-M m'}
    \big(1+|\xi|+|\eta|\big)^{-(\alpha_1+\alpha_2)}
\end{equation}
holds for all non-negative integers $\alpha_1, \alpha_2$ and $M$. 
By Fourier series we can expand the function as following. 
\begin{equation}
 \big({\Theta_1}(\xi){\wh{\Phi_6}}(\eta)\ti 
m_{6, j, \beta}(\xi, \eta)\big) = \sum_{n_1, n_2} 
C_{n_1, n_2}e^{2\pi i n_1\xi+2\pi i n_2\eta}\,,
\end{equation}
where the Fourier coefficients $C_{n_1, n_2}$'s satisfy
\begin{equation}\label{Fest31}
 |C_{n_1, n_2}|\leq C_{M, \beta}2^{-\beta Mj}2^{-M m'} 
(1+|n_1|)^{-M}(1+|n_2|)^{-M}
\end{equation} 
for all $M\geq 0$. A change of variables then yields
$$
 {\Theta_1}\big( \frac{\xi}{2^{-b_\beta +1+\beta j+j}}\big) 
{\wh{\Phi_6}}\big( \frac{\eta}{2^{m'+\beta j +2j }}\big) 
 m_{j,\beta}(\xi,\eta)
= \sum_{n_1, n_2} C_{n_1, n_2}e^{2\pi i n_12^{b_\beta-1-\beta j-j}\xi+2\pi i n_2 2^{-m'-\beta j -2j }\eta}\,,
$$ 
since $m_{j, \beta}(\xi, \eta)=
\ti m_{6, j, \beta}(\xi/2^{-b_\beta+1+\beta j+j}, \eta/2^{m'+
\beta j+2j})$. 
And hence $T_{j, \beta, 6}$ can be written as a product, i.e., 
$$
 T_{j, \beta, 6}(f, g)(x)=\sum_{n_1, n_2}\sum_{m'\geq b_\beta}C_{n_1, n_2}f_{n_2, j}(x) 
g_{n_1, j, m'}(x)\,,
$$
where 
$$
 {\wh f_{n_1, j}}(\xi) = \wh f(\xi) e^{2\pi in_1\xi/2^{-b_\beta+1+\beta j+j}}
  \Theta\big(\frac{\xi}{2^{-b_\beta+1+\beta j+j}} \big)\,
$$
$$
 {\wh g_{n_2, j, m'}}(\eta) = \wh g(\eta) e^{2\pi in_2\eta/2^{m'+\beta j+2j}}
  \wh\Phi\big(\frac{\eta}{2^{m'+\beta j+2j}} \big)\,
$$
(\ref{est137}) follows immediately from (\ref{Fest31}) because 
each term in the sum is trivially bounded. \\

The case $\ell=2$ can be obtained similarly by using Fourier series.
 The case $\ell=5$ is similar to the case $\ell=6$ by symmetry.
 We omit the details for these two cases. Therefore we finish the proof.
\end{proof}

\section{ Case $T_{j, \beta, 4}$ }\label{T4}
\setcounter{equation}0

From the definition of $T_{j, \beta, 4}$ and (\ref{defofphi0}), we have
that $T_{j, \beta, 4}(f,g)(x)$ equals to
\begin{equation}\label{defofT4}
 \sum_{-b_\beta < m< b_\beta  } T_{4, m, j, \beta}(f, g)(x)\,.
\end{equation}
where 
$$
 T_{4, m, j, \beta}(f, g)(x) =
\int\!\!\int
\wh f(\xi)\wh g(\eta)e^{i(\xi+\eta)x}
\wh\Phi\big(\frac{\xi}{2^{m+\beta j +j}} \big)
\Theta\big(\frac{\eta}{2^{
-b_\beta+1 +\beta j+2j}}\big) 
 m_{j, \beta}(\xi, \eta) d\xi d\eta\,.
$$
 
We need to show the following lemma. 

\begin{lemma}\label{lemT4}
Let $j\geq 0$, $\beta>1$ and $-b_\beta< m< b_\beta $.
There is a positive number $\e_0$ and  a constant $C$ such that 
\begin{equation}\label{est4}
 \big\|T_{4, m,j, \beta}(f, g)\big\|_2\leq C2^{-\e_0 j}\|f\|_\infty
\|g\|_2\,.
\end{equation}
holds for all $f\in L^\infty$ and $g\in L^2$.
\end{lemma}

In Lemma \ref{lemT4}, the positive number $\e_0$ can be chosen to be $(\beta -1)/5$.\\

Let $\ti m_{j, \beta}$ be defined by 
\begin{equation}
 \ti m_{j, \beta}(\xi, \eta) = m_{j, \beta}(2^j \xi, 2^j\eta)\,.
\end{equation}
Define $\ti T_{4, m, j, \beta}$ by 
\begin{equation}
\int\!\!\int
\wh f(\xi)\wh g(\eta)e^{i(\xi+\eta)x}
\wh\Phi\big(\frac{\xi}{2^{m+\beta j}} \big)
\Theta\big(\frac{\eta}{2^{
-b_\beta+1 +\beta j + j}}\big) 
 \ti m_{j, \beta}(\xi, \eta) d\xi d\eta\,.
\end{equation}

By a rescaling argument, to prove Lemma \ref{lemT4}, 
it is sufficient to show 
\begin{equation}\label{est44}
  \big\|\ti T_{4, m,j, \beta}(f, g)\big\|_2\leq C2^{-\e_0 j}\|f\|_\infty
\|g\|_2\,.
\end{equation}

\subsection{The tri-linear form}
To prove (\ref{est44}), we first reduce the problem to the $L^2$ estimate
of a tri-linear form.
Let $\La_{4, m, j, \beta}$ be the tri-linear form defined by
\begin{equation}\label{triform4}
 \La_{4, m, j, \beta}(f_1, f_2, f_3)
 = \int f_1(\xi) f_2(\eta) f_3 (\xi +\eta)
\wh\Phi\big(\frac{\xi}{ 2^{m+1 +\beta j }}\big)
   \Theta \big(\frac{\eta}{2^{-b_\beta +2 +\beta j+j}}\big)
  \ti m_{j, \beta}(\xi, \eta) d\xi d\eta\,.
\end{equation}

We claim that in order to prove (\ref{est44}), it is sufficient to prove 
\begin{equation}\label{T4222}
\big|\La_{4, m,j, \beta}(f_1, f_2, f_3)\big|\leq 
C 2^{-\e_0 j}\|f_1\|_2\|f_2\|_2\|f_3\|_2\,.
\end{equation}

Indeed,  
notice that $\ti T_{4, m, j, \beta}$ equals to
$$
\int f_{j, \beta}(x-t)g_{j, \beta}(x-2^{-j}t^2)
   \rho(t) e^{-i2^{\beta j}|t|^{-\beta}} dt\,,
$$
where $f_{j, \beta}$ and $g_{j, \beta}$ satisfy
$$
\wh f_{j, \beta}(\xi) = \wh f(\xi) \wh\Phi(\xi/2^{m + \beta j})\,,
$$
$$
\wh g_{j, \beta}(\eta)=\wh g(\eta)\Theta(\eta/ 2^{-b_\beta+1+\beta j +j })\,.
$$
By a similar estimate to (\ref{estlem}) and an interpolation, one can easily obtain 
\begin{equation}\label{est44p}
  \big\|\ti T_{4, m,j, \beta}(f, g)\big\|_2\leq C2^{-\ti\e_0 j}\|f\|_p
\|g\|_2\,,
\end{equation}
for all $p$ in $(1,\infty)$ and some positive number $\ti\e_0$.
However, we have to deal with the endpoint case $p=\infty$, which 
requires the following technical work.

Let $\psi$ be a non-negative Schwartz function such that $\wh \psi$
is supported in $[-1/100,1/100]$ and satisfies $\wh\psi(0)=1$. 
 And for $n\in \mathbb Z$, define
$$
I_{n} = [n, (n+1)]\,.
$$
Let $\Id_I$ be the characteristic function of the set $I$. 
Define
$$
\Id^*_{n}(x) = \Id_{I_{n}}*\psi(x)\,.
$$
$$
\Id^{**}_n(x) = \frac{1}{\big(1+ |x-n|^2\big)^{200}}
$$

It is clear that 
$$
\sum_n\Id^*_{n}(x)=1\,.  
$$
$\Id^*_{n}$ can be considered as essentially $\Id_{I_{n}}$. 
We thus can write 
$\langle \ti T_{4, m, j, \beta}(f, g), h \rangle$ as
$$
\int\!\!\int \sum_n \big(\Id^*_{n}f_{j, \beta}\big)(x-t)
 \sum_{n}\big(\Id^*_{n}g_{j, \beta}\big)(x-2^{-j}t^2)
    \sum_{n}\big(\Id^*_n h\big)(x)\rho(t) 
e^{-i2^{\beta j}|t|^{-\beta}}  dx\,dt,
$$
which is equal to
$$
\sum_{k_1}\sum_{k_2} \sum_n 
 \La_{k_1, k_2, n, m, j, \beta, 4}(f, g, h)\,,
$$
where
$
\La_{k_1, k_2, n, m, j, \beta, 4}(f, g, h)$ equals to
$$
\int\!\!\int \big(\Id^*_{n+k_1}f_{j, \beta}\big)(x-t)
 \big(\Id^*_{ n+k_2}g_{j, \beta}\big)(x-2^{-j}t^2)
\big(\Id^*_n h\big)(x)\rho(t) 
e^{-i2^{\beta j}|t|^{-\beta}}  dx\,dt\,.
$$
Let $\e$ be a small positive number. 
Putting absolute value throughout, we  
estimate the sum of $\La_{k_1, k_2, n, m, j, \beta, 4}(f, g, h)$ 
for all $(k_1, k_2, n)$'s with $\max\{|k_1|, |k_2|\}> 2^{\e j}$
 by 
$$
\sum_{\substack {(k_1, k_2): \\
 \max\{|k_1|, |k_2|\}> 2^{\e j}} }\!\!\!\!\!\!\!\!\!\!\!\!\!\!\sum_{n}
 \int \!\!\!\int   
 \frac{C_N|\big(\Id^{**}_{n+k_1}f_{j, \beta}\big)(x-t)|
|\big(\Id^{**}_{n+k_2}g_{j, \beta}\big)(x-2^{-j}t^2)|
|\big(\Id^*_n h\big)(x)  |
|\rho(t)| }{\big(1+ |k_1+t|\big)^N \big(1+ |k_2+2^{-j}t^2|\big)^N}
   dx dt \,,
$$
for all positive integers $N$. Notice that 
$t\sim 1$ when $t$ is in the support of $\rho$. Thus, 
for $\max\{|k_1|, |k_2|\} > 2^{\e j}$, we estimate this sum by 
$$
C_N2^{-N\e j}\|f\|_{\infty}\sum_{k_2}\frac{1}{(1+|k_2|)^N}
\sum_n \big\|\Id^{**}_{n+k_2}g_{j,\beta}\big\|_2
 \big\|\Id^{**}_n h\big\|_2
\leq C_N2^{-N\e j}\|f\|_\infty\|g\|_2\|h\|_2\,.
$$
We now turn to sum $\La_{k_1, k_2, n, m, j, \beta, 4}(f, g, h)$ for
all $|k_1|< 2^{\e j}$ and $|k_2|< 2^{\e j}$.
 Note that when $j$ is large, $\Id^*_{n+k_1}f_{j, \beta}$'s Fourier 
transform is supported in a small neighborhood of the support of 
$\wh{f_{j, \beta}}$. And $ \Id^*_{ n+k_2}g_{j, \beta}$ has a similar 
property.  Thus we have 
$$
\La_{k_1, k_2,n,  m, j, \beta, 4}(f, g, h) 
 = \La_{4, m, j,\beta} (\wh{\Id^*_{n+k_1}f_{j, \beta}}, 
 \wh{ \Id^*_{n+k_2} g_{j, \beta}}, \wh {\Id^*_n h})\,.
$$
And then (\ref{T4222}) yields
\begin{equation}\label{ksmall4}
\sum_{\substack{(k_1, k_2):\\ \max\{|k_1|, |k_2|\}\leq 2^{\e j}}}\!\!\!\!\!\!\!\!\!\!\!\!\sum_n
 |\La_{k_1, k_2,n,  m, j, \beta, 4}(f, g, h) |
\leq 2^{-\e_0 j}\!\!\!\!\!\!\!\!\!\!\!\!\!\!\!\!
\sum_{\substack{(k_1, k_2):\\ \max\{|k_1|, |k_2|\}\leq 2^{\e j}}}
\!\!\!\!\!\!\!\!\!\!\!\!\!
\sum_n \|\Id^*_{n+k_1}f_{j, \beta}\|_2
 \|\Id^*_{n+k_2} g_{j, \beta}\|_2
 \|\Id^*_n h  \|_2\,,
\end{equation}
which is clearly bounded by
$$
 C 2^{(2\e-\e_0)j}\|f\|_\infty\|g\|_2\|h\|_2\,.
$$
Since $\e$ can be chosen to be very small, we thus obtain 
(\ref{est4}) if (\ref{T4222}) is assumed to be right.
Therefore the remaining thing that we need to prove is (\ref{T4222})
for the boundedness of the operator $T_{j,\beta, 4}$. 
\\

Define $ m_{4, \beta, j}$ to be
$$
 m_{4, \beta, j}(\xi, \eta) =
\int\rho(t) e^{i2^{\beta j}\phi_{4, \xi, \eta}(t)} dt\,,
$$
where  
$$
\phi_{4,  \xi, \eta}(t)= 2^{m+1}\xi t + 
2^{-b_\beta+2}\eta t^2 - |t|^{-\beta}\,.
$$

Define the tri-linear form $\La_{j, \beta, m, 4}$ by
\begin{equation}\label{defLajb4}
\La_{j, \beta, m, 4}(f_1, f_2, f_3)=
\int\!\!\int  f_1(\xi) f_2(\eta) 
f_3\big(2^{ b_\beta+m-1-j}\xi+\eta\big)
\wh \Phi(\xi)\Theta(\eta)m_{4, \beta, j}(\xi, \eta)d\xi d\eta\,.
\end{equation}

By rescaling, to get (\ref{T4222}), it is sufficient to prove 
that there exists a positive number $\e_0$ such that 
\begin{equation}\label{triest4bj}
\big|\La_{j, \beta, m,  4}(f_1, f_2, f_3) \big|\leq 
 C_\beta 2^{-\beta j/2} 2^{-\e_0 j} \|f_1\|_2\|f_2\|_2\|f_3\|_2\,
\end{equation}
holds if $ -b_\beta< m< b_\beta $.\\

\subsection{The Stationary Phase}

For simplicity we suppose that $\rho$ is supported on $[1/8, 1/2]$.
When the support of $\rho$ is $[-1/2, -1/8]$, the same method works. 
The phase function $\phi_{4, \xi, \eta}$ satisfies the favorable estimate:
\begin{equation}\label{largeD24}
 \big| \phi''_{4, \xi, \eta}(t)\big| \geq C_\beta
\end{equation}
whenever $\eta$ in the support of $\Theta$ and 
 $t\in [1/16, 9/16]$  due to the definition of 
$b_\beta$. 
Thus $\phi'_{4, \xi, \eta}$ is monotone in $[1/16, 9/16]$.  
If in $[1/16, 9/16]$ 
there is no critical point of $\phi_{4, \xi, \eta}$, then 
(\ref{largeD24}) yields that
\begin{equation}\label{largeD14}
 \big| \phi'_{4, \xi, \eta}(t)\big| \geq C_\beta\,
\end{equation}
holds for all $t\in[1/8, 1/2]$. Integration by parts then gives
\begin{equation}\label{smallm4}
 |m_{4, \beta, j}(\xi, \eta)|\leq C_N2^{-N\beta j}
\end{equation}
for all positive integers $N$. Hence in this case, 
(\ref{triest4bj}) becomes trivial. \\

The difficult case is when there is a unique critical point of 
$\phi_{4, \xi, \eta}$ in $[1/16, 9/16]$. Let us call this critical point
$t_0=t_0(\xi, \eta)$. 
The method of stationary phase yields that
$$m_{4, \beta, j}(\xi, \eta)
 \sim
 \frac{Ce^{i 2^{\beta j}\phi_{4, \xi, \eta}(t_0)}}{2^{\beta j/2}} \,,
$$
since we have
(\ref{largeD24}) and a trivial upper bound when $\xi, \eta$
are in the supports of $\wh\Phi $ and $\Theta $ respectively.
The stationary phase gives a high oscillation, that is, 
the phase $ 2^{\beta j}e^{\phi_{4, \xi, \eta}(t_0)}$ 
causes $e^{i2^{\beta j}\phi_{4, \xi, \eta}}$ to be 
a highly oscillatory factor whenever $\xi\in\operatorname{supp}\wh{\Phi}$ and
$\eta\in\operatorname{supp}\Theta$. 
 And we will see that this high oscillation 
yields a desired estimate. To prove (\ref{triest4bj}), it is
enough to show the following lemma.

\begin{lemma}\label{osc4bj}
Let $-b_\beta < m < b_\beta$. And 
let $\ti \La_{j, \beta, m, 4}$ be defined by 
\begin{equation}\label{defLajb40}
\ti \La_{j, \beta, m, 4}(f_1, f_2, f_3)=
\int \!\!\int f_1(\xi) f_2(\eta) f_3\big(2^{b_\beta+m-1-j}\xi+\eta\big)
\wh \Phi(\xi)\Theta(\eta) e^{i 2^{\beta j}\phi_{4,\xi, \eta}(t_0)}d\xi d\eta\,.
\end{equation}
Suppose that $\beta > 1$. Then there exist a positive number $\e_0$ 
and a constant $C_\beta$ independent of $j$
such that
\begin{equation}\label{triest4bj0}
\big|\ti\La_{j, \beta, m, 4}(f_1, f_2, f_3) \big|\leq 
 C_\beta 2^{-\e_0 j} \|f_1\|_4\|f_2\|_2\|f_3\|_2\,
\end{equation}
holds for all functions $f_1\in L^4$ and $f_2, f_3\in L^2$. 
\end{lemma}

We now can see that (\ref{triest4bj}) follows from this Lemma. Indeed, 
it is easy to get a trivial estimate by inserting 
absolute values throughout and Cauchy-Schwarz inequality:
\begin{equation}\label{4est122}
 \big|\ti \La_{j, \beta, m, 4}(f_1, f_2, f_3)\big|\leq 
C\|f_1\|_1\|f_2\|_2\|f_3\|_2\,.
\end{equation}
Thus if $\beta >1$, then by an interpolation, (\ref{triest4bj0}) and (\ref{4est122}) yields 
\begin{equation}\label{444est222}
 \big|\ti \La_{j, \beta, m, 4}(f_1, f_2, f_3)\big|\leq 
C 2^{-\e'_0 j}\|f_1\|_2\|f_2\|_2\|f_3\|_2\,,
\end{equation}
for some $\e'_0>0$, which gives (\ref{triest4bj}) immediately.


 The rest of the section is devoted to the proof of (\ref{triest4bj0}).\\

\subsection{Some lemmata}
We need some lemmata for the proof of (\ref{triest4bj0}). 

\begin{lemma}\label{lemHessian0}
Let $\phi(t, \xi, \eta)= a\xi t + b \eta t^2 + f(t)$ for some 
$C^\infty$ function $f$ and $a, b\in\mathbb R$. 
Let $t_0(\xi, \eta)$ be a critical 
point of $\phi(\cdot, \xi, \eta)$ such that 
$$
 \phi''(t_0(\xi, \eta), \xi, \eta)\neq 0\,, 
$$  
where $\phi''$ is the second order derivative with respect to $t$.
Define
\begin{equation}\label{defofQ}
 Q(\xi, \eta)= \phi(t_0(\xi, \eta), \xi, \eta)\,.
\end{equation}
Then the determinant of the Hessian matrix of $Q$ vanishes. 
\end{lemma}

\begin{proof}
$t_0(\xi, \eta)$ is implicitly defined by the equation
$$
 a\xi + 2b \eta t + f'(t)=0\,.
$$
Thus we have 
\begin{equation}\label{tpxi}
\frac{\p t_0}{\p \xi} = \frac{-a}{2b\eta + f''(t_0)} =
 \frac{-a}{\phi''(t_0(\xi, \eta), \xi, \eta)} \,.
\end{equation}
\begin{equation}\label{tpeta}
\frac{\p t_0}{\p \eta} = \frac{-2bt_0}{2b\eta + f''(t_0)} =
 \frac{-2bt_0}{\phi''(t_0(\xi, \eta), \xi, \eta)} \,.
\end{equation}
By the chain rule and the fact that $t_0(\xi, \eta)$ is a critical point, 
we have
\begin{equation}
 \frac{\p Q}{\p \xi} = a t_{0}(\xi, \eta) \,.  
\end{equation} 
\begin{equation}
 \frac{\p Q}{\p \eta} = b t_{0}^2(\xi, \eta) \,.  
\end{equation} 
Thus 
$$
 \frac{\p^2 Q}{\p \xi^2} \frac{\p^2 Q}{\p \eta^2} = 2ab t_0
     \frac{\p t_0}{\p \xi} \frac{\p t_0}{\p \eta}
 = \bigg(\frac{2abt_0}{\phi''(t_0, \xi, \eta)}\bigg)^2\,.
$$
$$
  \frac{\p^2 Q}{\p \xi\p\eta} = a\frac{\p t_0}{\p \eta}
 = \frac{-2ab t_0}{\phi''(t_0, \xi, \eta)}\,.
$$
Clearly, the determinant of the Hessian matrix of $Q$ vanishes.
\end{proof}

\begin{lemma}\label{4Hessian2}
Let $t_0$ be a critical point of $\phi_{4, \xi, \eta}$.
Define $\bQ$ by 
\begin{equation}\label{defbQ}
 \bQ(\xi, \eta) =  \phi_{4, \xi, \eta}(t_0)
\end{equation}
Let $j>0$, $ |\tau | \leq C$, $|\alpha|\leq C_\beta$, 
$(u, v)\in {\rm supp}{\wh\Phi}\times {\rm supp}{\Theta}$. 
Suppose that
$t_0(u, v), t_0(u-\tau, v+\alpha 2^{-j}\tau)\in [1/16, 9/16]$
exist. 
And let $\ti\bQ_\tau$
be defined by
$$
 \ti\bQ_\tau(u, v) = \bQ(u, v) -\bQ(u-\tau, v+\alpha 2^{-j}\tau)\,.
$$
If $j$ is large enough (larger than a constant), 
then the determinant of the Hessian matrix 
of $\ti\bQ_\tau$ satisfies
\begin{equation}\label{bigHessian2} 
 \big|{\rm det} H(\ti\bQ_\tau)\big|\geq C\tau^2\,,
\end{equation}
where $H(\ti\bQ_{\tau})$ denotes the Hessian matrix. 
\end{lemma}

\begin{proof}
Using Lemma \ref{lemHessian0}, it is easy to see that 
the determinant of the Hessian of $\ti Q_\tau$ is equal to
$$
2 \bC(u, v)\bC(u-\tau, v+\alpha 2^{-j}\tau) - 
\bA(u, v)\bB(u-\tau, v+\alpha 2^{-j}\tau)
- \bA(u-\tau, v+\alpha 2^{-j}\tau)
\bB(u, v)\,.
$$ 
where $\bC = \frac{\p^2 \bQ}{\p u\p v}$, 
$\bA=\frac{\p^2\bQ}{\p u^2}$ and 
$\bB = \frac{\p^2\bQ}{\p v^2}$.

Let $\bD_2(u, v)= \phi''_{4, u, v}(t_0(u, v))$. A simple computation
as we did in Lemma \ref{lemHessian0} then yields that
$$
 {\rm det} H(\ti\bQ_\tau)= \frac{-4\cdot 2^{2\ti b_\beta+2m'}}{\bD_2(u, v)\bD_2
(u-\tau, v+\alpha 2^{-j}\tau) } \big( t_0(u,v)-t_0(u-\tau, v+\alpha 
2^{-j}\tau)\big)^2\,.
$$
It is easy to see that 
\begin{equation}\label{D24}
 \phi''_{4, \xi, \eta}(t)
 = 2^{-b_\beta+3}\eta -\frac{\beta(\beta+1)}{t^{\beta+2}}\,.
\end{equation}
Let $2{\rm supp}\Theta$ be an interval generated by dilating the interval ${\rm supp}\Theta$ into an interval with twice length. For all $\eta\in 2{\rm supp}\Theta$ and $t\in[1/32, 19/32]$, we
have 
\begin{equation}\label{D241}
 \big|\phi''_{4, \xi, \eta}(t)\big|\sim C_\beta\,
\end{equation}
due to the definition of $b_\beta$.
If $j$ is large enough, then $v+\alpha 2^{-j}\tau\in 2{\rm supp}\Theta$
since $v\in{\rm supp}\Theta$. (\ref{D241}) then yields
$$
 |\bD_2(u, v)|\sim C_\beta, \,\,\, {\rm and}\,\,\,
|\bD_2(u-\tau, v+\alpha 2^{-j}\tau)|\sim C_\beta\,.
$$
Thus to finish the proof it is sufficient to show that
\begin{equation}\label{diff41}
 \big| t_0(u,v)-t_0(u-\tau, v+\alpha 2^{-j}\tau)\big| \geq C\tau\,.
\end{equation}

We claim first that there is a critical point of $\phi_{4, u-\tau, v}$
 in $[1/32, 19/32]$. In other words, this means that 
$t_0(u-\tau, v)\in [1/32, 19/32]$ exists. We prove this claim by 
contradiction. Assume that such a critical point does not exist, 
that is, 
$$
\phi'_{4, u-\tau, v}(t)=2^{m+1}(u-\tau)+2^{-b_\beta+3}vt+\frac{\beta}{t^{\beta +1}}\neq 0\,
$$
for all $t\in[1/32, 19/32]$. (\ref{D241}) then gives that
\begin{equation}\label{largeD144}
\big|\phi'_{4, u-\tau, v}(t) \big| \geq C_\beta
\end{equation}
for all $t\in [1/16, 9/16]$, because $\phi''$ does not change 
sign in $[1/32, 19/32]$ and then $\phi'$ is monotonic in
$[1/32, 19/32]$.  However, $t_0(u-\tau, v+\alpha 2^{-j}\tau)\in 
 [1/16, 9/16]$ exists. This yields a contradiction if 
$j$ is large enough, 
since 
$$
\big|\phi'_{4, u-\tau, v}(t) - \phi'_{4,u-\tau, v+\alpha 2^{-j}\tau}(t)\big|
\leq C_\beta 2^{-j}\,
$$
holds for all $t\in [1/16, 9/16]$. Thus we know that
$t_0(u-\tau, v)\in [1/32, 19/32]$ must exist. \\

The second claim we try to make is that for any $\theta\in [0,1]$, 
$t_0( u-\theta\tau, v)\in [1/32, 19/32]$ exists, and 
$t_0(u-\tau, v+\theta \alpha 2^{-j}\tau )\in[1/32, 19/32]$ exists.
Indeed,  notice that
$$
\frac{\p}{\p \xi}(\phi'_{4, \xi, v}(t)) = 2^{m+1} > 0\,.
$$ 
Thus $\phi'_{4, \xi, v}(t)$ is an increasing function in $\xi$.
If there exists $\xi\in (u, u-\tau)$ (we assume $\tau <0$ here, the 
another case $\tau>0$ is similar) such that 
$\phi'_{4, \xi, v}(t)\neq 0 $ for all $t\in [1/32, 19/32]$.
Then $\phi'_{4, \xi, v}(t)>0$ or $\phi'_{4, \xi, v}(t)<0 $
for all $t\in [1/32, 19/32]$. 
When $\phi'_{4, \xi, v}(t)>0$, we have 
$ \phi'_{4, u-\tau, v}(t)>0 $ for all $t\in[1/32, 19/32]$, 
which is a contradiction to the existence of $t_0(u-\tau, v)\in [1/32, 19/32]$.
When $\phi'_{4, \xi, v}(t)<0$, we have 
$ \phi'_{4, u, v}(t)<0 $ for all $t\in[1/32, 19/32]$, 
which contradicts to the existence of $t_0(u, v)\in [1/32, 19/32]$.
Thus $t_0( u-\theta\tau, v)\in [1/32, 19/32]$ exists. 
A similar argument yields that $t_0(u-\tau, v+\theta \alpha 2^{-j}\tau )\in[1/32, 19/32]$ exists.\\

We now turn to prove (\ref{diff41}). The triangle inequality 
yields that the left hand side of (\ref{diff41}) is bigger than or equal
to 
\begin{equation}\label{diff4444}
\big|t_0(u, v)-t_0(u-\tau, v)\big| -
\big|t_0(u-\tau, v)-t_0(u-\tau, v+\alpha 2^{-j}\tau)\big|\,.
\end{equation}
By the mean value theorem, we have 
\begin{equation}\label{mean41}
 \big|t_0(u-\tau, v)-t_0(u-\tau, v+\alpha2^{-j}\tau)\big|
\leq  C2^{-j}\tau\bigg |\frac{\p t_0}{\p v} (u-\tau, v+\theta\alpha 2^{-j}\tau )  \bigg|\,, 
\end{equation}
for some $\theta\in[0, 1]$. From (\ref{tpeta}), 
it is easy to see that
$$
\bigg |\frac{\p t_0}{\p v} (u-\tau, v+\theta\alpha 2^{-j}\tau )  \bigg|
 = \frac{2^{-b_\beta+3}t_0(u-\tau, v+\theta\alpha2^{-j}\tau)}{
 \big|\phi''_{4, u-\tau, v+\theta\alpha 2^{-j}\tau}\big(t_0(u-\tau, v+\theta \alpha2^{-j}\tau)\big) \big|} \leq C_\beta\,.
$$
\begin{equation}\label{mean411}
 \big|t_0(u-\tau, v)-t_0(u-\tau, v+\alpha 2^{-j}\tau)\big|\leq  C2^{-j}\tau\,.
\end{equation}
Similarly, the mean value theorem and (\ref{tpeta}) also yield 
\begin{equation}\label{mean412}
 \big|t_0(u, v)-t_0(u-\tau, v)\big|\geq  C\tau \,.
\end{equation}
(\ref{mean411}) and (\ref{mean412}) then give (\ref{diff41}).
Therefore we finish the proof of the lemma. 
\end{proof}

In the proof of Lemma \ref{4Hessian2}, we proved the stability of 
the critical point $t_0(u, v)$. 
We now are ready to prove that $\tilde{\bQ}_\tau$ is not degenerate. 

\begin{lemma}\label{4nondeg}
Let $\tilde{\bQ}_\tau$ be the function defined as in Lemma \ref{4Hessian2}.
If $j$ is large enough, then
\begin{equation}\label{nondeg4}
 \bigg|\frac{\partial^2 \tilde{\bQ}_\tau}{\partial u\partial v}(u, v)\bigg| \geq C_\beta\tau\,
\end{equation}
\begin{equation}\label{nondeg41}
 \bigg|\frac{\partial^3 \tilde{\bQ}_\tau}{\partial^2 u\partial v}(u, v)\bigg| \geq C_\beta\tau\,
\end{equation}
hold for all $(u, v)\in {\rm supp}\wh\Phi\times {\rm supp}\Theta$.
\end{lemma}
\begin{proof}
Clearly 
$$
\frac{\partial^2 \tilde{\bQ}_\tau}{\partial u\partial v}(u, v)
=\frac{\partial^2 {\bQ}}{\partial u\partial v}(u, v)
-\frac{\partial^2 {\bQ}}{\partial u\partial v}(u-\tau, v+\alpha 2^{-j}\tau)\,.
$$
We can estimate $|\frac{\partial^2 \tilde{\bQ}_\tau}{\partial u\partial v} |$ by
$$
\bigg|\frac{\partial^2 {\bQ}}{\partial u\partial v}(u, v) -
\frac{\partial^2 {\bQ}}{\partial u\partial v}(u-\tau, v)\bigg|
+\bigg|\frac{\partial^2 {\bQ}}{\partial u\partial v}(u-\tau, v)-
\frac{\partial^2 {\bQ}}{\partial u\partial v}(u-\tau, v+\alpha 2^{-j}\tau)\bigg|\,.
$$
By the mean value theorem, the first term in the previous sum is majorized by
$$
\bigg|\frac{\partial^3 {\bQ}}{\partial^2 u\partial v}(u-\theta\tau, v)\bigg|\tau
$$
for some $\theta\in [0,1]$.
From the proof of Lemma \ref{lemHessian0}, we have 
\begin{equation}\label{4dudv}
\frac{\partial^2 {\bQ}}{\partial u\partial v}(u,v)=
\frac{- 2^{m+1} 2^{-b_\beta+3}t_0(u, v)}{\phi''_{4,u,v}(t_0)}\,.
\end{equation} 
Thus 
\begin{equation}\label{4d2udv}
\frac{\partial^2 {\bQ}}{\partial^2 u\partial v}(u,v)=
\frac{- 2^{m+1} 2^{-b_\beta+3} 
 \frac{\partial t_0}{\partial u}(u, v) \big(\phi''_{4, u, v}(t_0)
 - \beta(\beta+1)(\beta+2)|t_0|^{-\beta-2} \big)
}{\big(\phi''_{4,u,v}(t_0)\big)^2}\,,
\end{equation}
which is equal to 
$$
 \frac{ - 2^{m+1} 2^{-b_\beta+3} 
 \frac{\partial t_0}{\partial u}(u, v) \big(
 {\ti C}_\beta |t_0|^{-\beta-2} + 2^{-b_\beta+3} v\big)
}{\big(\phi''_{4,u,v}(t_0)\big)^2}\,,
$$
where ${\ti C}_\beta $ is a number such that $|{\ti C}_\beta|> \beta(\beta+1)^2$.
Since $b_\beta$ was chosen to be a large number and $|v|\leq C$ 
whenever $v\in {\rm supp}\Theta$, we have 
 \begin{equation}\label{4d2udv1}
\bigg|\frac{\partial^2 {\bQ}}{\partial^2 u\partial v}(u,v)\bigg|\geq 
 C_\beta   \,
\end{equation}
for all $(u, v)\in {\rm supp}\wh\Phi\times {\rm supp}\Theta  $.
Thus we obtain 
\begin{equation}\label{e1stterm}
\bigg|\frac{\partial^2 {\bQ}}{\partial u\partial v}(u, v) -
\frac{\partial^2 {\bQ}}{\partial u\partial v}(u-\tau, v)\bigg|
\geq C_\beta \tau\,,
\end{equation}
for all all $(u, v)\in {\rm supp}\wh\Phi\times {\rm supp}\Theta  $.\\

Notice that $\frac{\partial^2 {\bQ}}{\partial u\partial^2 v}(u,v)$ equals to 
$$
 \frac{- 2^{m+1} 2^{-b_\beta+3}  \big(
 \frac{\partial t_0}{\partial v}(u, v)\phi''_{4, u, v}(t_0)
 - \frac{\partial t_0}{\partial v}(u, v)\beta(\beta+1)(\beta+2)|t_0|^{-\beta-2} 
- 2^{-b_\beta+3} \big)
}{\big(\phi''_{4,u,v}(t_0)\big)^2}\,,
$$
which is clearly bounded by $C_\beta$.
The mean value theorem then yields 
\begin{equation}\label{e2ndterm}
\bigg|\frac{\partial^2 {\bQ}}{\partial u\partial v}(u-\tau, v)-
\frac{\partial^2 {\bQ}}{\partial u\partial v}(u-\tau, v+\alpha 2^{-j}\tau)\bigg|
\leq C_\beta 2^{-j}\tau\,.
\end{equation}
From (\ref{e1stterm}) and (\ref{e2ndterm}), we have (\ref{nondeg4}) if $j$ is large enough.
(\ref{nondeg41}) can be proved similarly. We omit the details. 

\end{proof}

\subsection{Proof of Lemma \ref{osc4bj}}
We now prove Lemma \ref{osc4bj}. 
Let $b_1=1-2^{b_\beta+m-1-j}$ and $b_2=2^{b_\beta+m-1-j}$. 
Changing variable 
$\xi\mapsto \xi-\eta$ and $\eta\mapsto b_1\xi+b_2\eta$, 
we have that
$$
\ti\La_{j, \beta, m, 4}(f_1,\! f_2,\! f_3)=\!\!\!\!
\int\!\!\!\int\!\! f_1(\xi-\eta) f_2(b_1\xi+b_2\eta) f_3(\xi)
\wh\Phi(\xi-\eta)\Theta(b_1\xi+b_2\eta) e^{i2^{\beta j}
\phi_{4, \xi-\eta, b_1\xi+b_2\eta }(t_0)} d\xi d\eta.
$$
Thus by Cauchy-Schwarz we dominate $|\ti\La_{j, \beta, m, 4}|$
by
$$
 \big\| \bT_{j, \beta, m, 4}(f_1, f_2)\big\|_2\|f_3\|_2\,,
$$
where $\bT_{j, \beta, m, 4}$ is defined by
$$
\bT_{j, \beta, m, 4}(f_1, f_2)(\xi)
= \int\!\! f_1(\xi-\eta) f_2(b_1\xi+b_2\eta) 
\wh\Phi(\xi-\eta)\Theta(b_1\xi+b_2\eta) e^{i2^{\beta j}
\phi_{4, \xi-\eta, b_1\xi+b_2\eta }(t_0)} d\eta.
$$ 
It is easy to see that
 $\big\| \bT_{j, \beta, m, 4}(f_1, f_2)\big\|_2^2$ equals to
$$
\int\bigg(\int\!\int F(\xi, \eta_1, \eta_2) G(\xi, \eta_1, \eta_2)
e^{ i 2^{\beta j} \big(\phi_{4,\xi-\eta_1, b_1\xi+b_2\eta_1 }(t_0)-
\phi_{4, \xi-\eta_2, b_1\xi+b_2\eta_2 }(t_0) \big)}
  d\eta_1d\eta_2 \bigg)  d\xi\,, 
$$
where  
$$
F(\xi, \eta_1, \eta_2)= \big(f_1\wh\Phi\big)(\xi-\eta_1 )
\overline{\big(f_1\wh\Phi\big)(\xi-\eta_2)  }
$$
$$
G(\xi, \eta_1,\eta_2)=\big(f_2\Theta\big)(b_1\xi+b_2\eta_1 )
 \overline{\big(f_2\Theta\big)(b_1\xi+b_2\eta_2)}\,.
$$
Changing variables $\eta_1\mapsto \eta$ and $\eta_2\mapsto \eta+\tau$, 
we see that $\big\|\bT_{j, \beta, m, 4}(f_1, f_2) \big\|_2^2$ equals to
$$
\int\!\bigg(\int\!\!\int F_\tau(\xi-\eta) G_{\tau}(b_1\xi+b_2\eta)
 e^{i 2^{\beta j}\big(\phi_{4,\xi-\eta, b_1\xi+b_2\eta }(t_0)-
\phi_{4, \xi-\eta-\tau, b_1\xi+b_2(\eta+\tau) }(t_0)
\big)}  d\xi d\eta                             \bigg)d\tau\,, 
$$
where 
$$
F_\tau(\cdot) = \big(f_1\wh\Phi\big)(\cdot )
\overline{\big(f_1\wh\Phi\big)(\cdot -\tau )  }
$$
$$
G_\tau(\cdot) = \big(f_2\Theta\big)(\cdot )
\overline{\big(f_2\Theta\big)(\cdot + b_2\tau )  }\,.
$$
Changing coordinates to $(u,v)=(\xi-\eta, b_1\xi+b_2\eta)$, 
the inner integral becomes
$$
\int\!\int F_\tau(u)G_\tau(v) e^{i2^{\beta j}
 \ti\bQ_\tau (u, v)  } du dv\,,
$$
where $\ti\bQ_{\tau}$ is defined by 
$$
\ti\bQ_\tau(u, v)=\bQ(u, v)-\bQ(u-\tau, v+b_2\tau)\,.
$$

Lemma \ref{4nondeg}
and the well-known H\"ormander theorem on the non-degenerate phase 
\cite{Hor, PSt1} 
yield  that  
$\big\| \bT_{j, \beta, m, 4}(f_1, f_2)\big\|_2^2$ is estimated  by 
$$
C\int_{-10}^{10} \min\big\{1, 2^{-\beta j/2}\tau^{-1/2 } \big\} 
\big\|F_\tau\big\|_2 \big\|G_\tau \big\|_2   d\tau\,.
$$
By Cauchy-Schwarz inequality it is bounded by
$$
 C_\e 2^{-\beta j(1-\e)/2} 2^{j/2}\|f_1\|_4^2\|f_2\|_2^2\,,
$$
for any $\e>0$.  Thus we have 
\begin{equation}\label{4est422}
 \big|\ti \La_{j, \beta, m, 4}(f_1, f_2, f_3)\big|\leq 
C_\e 2^{-\beta j(1-\e)/4} 2^{j/4}\|f_1\|_4\|f_2\|_2\|f_3\|_2\,.
\end{equation}
Taking $\e_0 = (\beta-1)/5$, we then have 
\begin{equation}\label{4est422final}
 \big|\ti \La_{j, \beta, m, 4}(f_1, f_2, f_3)\big|\leq 
C_\beta 2^{-\e_0 j}\|f_1\|_4\|f_2\|_2\|f_3\|_2\,.
\end{equation}
This completes the proof of Lemma \ref{osc4bj}.


\section{Cases $T_{j, \beta, 7} $ and $T_{j, \beta, 8}$}\label{T8}
\setcounter{equation}{0}

$T_{j, \beta, 7}$ is similar to $T_{j, \beta, 8}$. We only give
the details for $T_{j, \beta, 8}$. Recall that 
$T_{j, \beta, 8}=\sum_{m, m'\geq b_\beta} T_{m, m',j, \beta}$,
where 
$$
 T_{m, m', j, \beta}(f, g)(x)=\int\!\int \wh f(\xi)\wh g(\eta)
 e^{i(\xi+\eta)x}\wh\Phi(\frac{\xi}{2^{m+\beta j +j}})
\wh\Phi(\frac{\eta}{2^{m'+\beta j +2j}})m_{j, \beta}(\xi, \eta)d\xi d\eta\,.
$$
Notice that if $|m-m'|\geq 100$ and $m, m'\geq b_\beta$, then when 
$\xi\in{\rm supp}\wh\Phi(\cdot/2^{m+\beta j+j})$,
$\eta\in{\rm supp}\wh\Phi(\cdot/2^{m'+\beta j +2j})$ and 
$t\in{\rm supp}\rho$, 
the phase function 
$$
 \phi_{j, \beta}(t, \xi, \eta) =2^{-j}\xi t +2^{-2j}\eta t^2 -2^{\beta j}|t|^{-\beta}\,
$$
satisfies
\begin{equation}\label{largeD18}
 \big|\phi'_{j, \beta}(t, \xi, \eta)\big|\geq C_\beta \max\{2^{m}, 2^{m'}\}
  2^{\beta j}\,.
\end{equation}
Thus, as usual,  integration by parts then yields 
$$
|m_{j, \beta}(\xi, \eta)|\leq C_{\beta, N} 2^{-N\beta j} \min\{
2^{-Nm}, 2^{-Nm'}\}\,,
$$
Clearly this reduces the problem to the simplest case as we did 
in Section {\ref{product}}. Thus we only need to consider the 
case when $|m-m'|\leq 100$.  The main lemma is the following for this case. 

\begin{lemma}\label{lemT8} 
Let $\beta>1$, $j\geq 0$, $m', m\geq b_\beta$ 
and $|m-m'|\leq 100$. Then there is a positive number $\e_0$ and a constant $C$ such that
\begin{equation}\label{Tmm8}
\big\| T_{m, m', j, \beta}(f, g)\big\|_2\leq C2^{-\e_0 {(j+m)}}\|f\|_\infty\|g\|_2\, 
\end{equation}
holds for all $f\in L^\infty$ and $g\in L^2$. 
\end{lemma}

Now it is clear that the boundedness of $\sum_j T_{j,\beta, 8}$ follows by Lemma 
{\ref{lemT8}}. Hence we only need to show Lemma {\ref{lemT8}}.\\ 

Since in this case for fixed $m$ there
are only finitely  many $m'$'s, without loss of generality, we 
can assume $m=m'$ when we prove Lemma \ref{lemT8}. By rescaling,
to prove (\ref{lemT8}), it is sufficient to show
\begin{equation}\label{}
 \big\| \ti T_{m, j, \beta}(f, g)\big\|_2\leq C2^{-\e_0 {(j+m)}}\|f\|_\infty\|g\|_2\,
\end{equation}
where
$$
\ti T_{ m, j, \beta}(f, g)(x)=\int\!\int \wh f(\xi)\wh g(\eta)
 e^{i(\xi+\eta)x}\wh\Phi(\frac{\xi}{2^{m+\beta j }})
\wh\Phi(\frac{\eta}{2^{m+\beta j +j}})m_{j, \beta}(2^j\xi, 2^j\eta)d\xi d\eta\,.
$$
By a similar argument as we did in Section \ref{T4}, we can reduce the problem 
to show the $L^2$ boundedness of the following tri-linear form:
$$
\La_{8, m, j, \beta}(f_1, f_2, f_3)
 =\!\!\int\!\!\int\!\! f_1(\xi)f_2(\eta)f_3(\xi+\eta)
\wh\Phi\big(\frac{\xi}{2^{m+1+\beta j}}\big)
\wh\Phi\big(\frac{\eta}{2^{m+1+\beta j +j}}\big)m_{j, \beta}(2^j\xi, 2^j\eta)
d\xi d\eta\,,
$$
i.e. 
\begin{equation}\label{T8222}
\big|\La_{8, m, j, \beta}(f_1, f_2, f_3)\big|
\leq C2^{-\e_0 (j+m)}\|f_1\|_2\|f_2\|_2\|f_3\|_2\,.
\end{equation}

Define $ m_{8, \beta, j}$ to be
$$
 m_{8, \beta, j}(\xi, \eta) =
\int\rho(t) e^{i2^{\beta j+m+1}\phi_{8, \xi, \eta}(t)} dt\,,
$$
where  
$$
\phi_{8,  \xi, \eta}(t)= \xi t + 
\eta t^2 - 2^{-m-1}|t|^{-\beta}\,.
$$

Define the tri-linear form $\La_{j, \beta, 8}$ by
\begin{equation}\label{defLajb8}
\La_{j, \beta, m, 8}(f_1, f_2, f_3)=
\int\!\!\int  f_1(\xi) f_2(\eta) 
f_3\big(2^{-j}\xi+\eta\big)
\wh \Phi(\xi)\wh\Phi(\eta)m_{8, \beta, j}(\xi, \eta)d\xi d\eta\,.
\end{equation}

As the case $T_{j, \beta, 4}$,
by rescaling, to get (\ref{T8222}), it is sufficient to prove 
that there exists a positive number $\e_0$ such that 
\begin{equation}\label{triest8bj}
\big|\La_{j, \beta, m,  8}(f_1, f_2, f_3) \big|\leq 
 C_\beta 2^{-(\beta j+m)/2} 2^{-\e_0 (j+m)} \|f_1\|_2\|f_2\|_2\|f_3\|_2\,
\end{equation}
holds if $ m\geq b_\beta $ and $\beta>1$.\\

\subsection{The Tri-linear Oscillatory Integral}
As before, for simplicity,  we suppose that $\rho$ is supported on $[1/8, 1/2]$. 
We have the following favorable estimate for the phase function $\phi_{8, \xi, \eta}$.
\begin{equation}\label{largeD28}
 \big| \phi''_{8, \xi, \eta}(t)\big| \geq C_\beta\,,
\end{equation}
whenever $m\geq b_\beta$, $\eta$ is in the support of $\wh\Phi$ and 
 $t\in [1/16, 9/16]$  due to the definition of $b_\beta$. 
Thus $\phi'_{8, \xi, \eta}$ is monotone in $[1/16, 9/16]$.  
If in $[1/16, 9/16]$ 
there is no critical point of $\phi_{8, \xi, \eta}$, then 
(\ref{largeD28}) yields that
\begin{equation}\label{largeD18}
 \big| \phi'_{8, \xi, \eta}(t)\big| \geq C_\beta\,
\end{equation}
holds for all $t\in[1/8, 1/2]$. Integration by parts then gives
\begin{equation}\label{smallm8}
 |m_{8, \beta, j}(\xi, \eta)|\leq C_N2^{-N(\beta j+m)}
\end{equation}
for all positive integers $N$, which trivializes (\ref{triest8bj}). \\

The difficult case is when there is a unique critical point of 
$\phi_{8, \xi, \eta}$ in $[1/16, 9/16]$. Let us call this critical point
$t_0=t_0(\xi, \eta)$. 
The method of stationary phase yields that
$$m_{8, \beta, j}(\xi, \eta)
 \sim
 \frac{Ce^{i 2^{\beta j+m+1}\phi_{8, \xi, \eta}(t_0)}}{2^{(\beta j+m)/2}} \,,
$$
since we have
(\ref{largeD28}) and a trivial upper bound when $\xi, \eta$
are in the supports of $\wh\Phi$.
 The high oscillation from the stationary phase should
yield a desired estimate for us. To prove (\ref{triest8bj}), it is
enough to show the following lemma.

\begin{lemma}\label{osc8bj}
Let $ m \geq b_\beta$. And 
let $\ti \La_{j, \beta, m, 8}$ be defined by 
\begin{equation}\label{defLajb80}
\ti \La_{j, \beta, m, 8}(f_1, f_2, f_3)=
\int \!\!\int f_1(\xi) f_2(\eta) f_3\big(2^{-j}\xi+\eta\big)
\wh \Phi(\xi)\wh\Phi(\eta) e^{i 2^{\beta j+m+1}\phi_{8,\xi, \eta}(t_0)}d\xi d\eta\,.
\end{equation}
Suppose that $\beta > 1$. Then there exist a positive number $\e_0$ 
and a constant $C_\beta$ independent of $j$
such that
\begin{equation}\label{triest8bj0}
\big|\ti\La_{j, \beta, m, 8}(f_1, f_2, f_3) \big|\leq 
 C_\beta 2^{-\e_0 (j+m)} \|f_1\|_4\|f_2\|_2\|f_3\|_2\,
\end{equation}
holds for all functions $f_1\in L^4$ and $f_2, f_3\in L^2$. 
\end{lemma}

We now show that (\ref{triest8bj}) is a simple consequence of this Lemma. 
Indeed, 
it is easy to  obtain a trivial estimate by inserting 
absolute values throughout and Cauchy-Schwarz inequality:
\begin{equation}\label{8est122}
 \big|\ti \La_{j, \beta, m, 8}(f_1, f_2, f_3)\big|\leq 
C\|f_1\|_1\|f_2\|_2\|f_3\|_2\,.
\end{equation}
Thus if $\beta >1$, then by an interpolation, (\ref{triest8bj0}) and (\ref{8est122}) yields 
\begin{equation}\label{888est222}
 \big|\ti \La_{j, \beta, m, 8}(f_1, f_2, f_3)\big|\leq 
C 2^{-\e'_0 (j+m)}\|f_1\|_2\|f_2\|_2\|f_3\|_2\,,
\end{equation}
for some $\e'_0>0$, which gives (\ref{triest8bj}) immediately.

\subsection{Lemmas}
As in the case $T_{j, \beta, 4}$, we need the stability of the critical 
points of the phase function.

\begin{lemma}\label{8Hessian2}
Let $m\geq b_\beta$.
And let $t_0$ be a critical point of $\phi_{8, \xi, \eta}$.
Define $\bQ$ by 
\begin{equation}\label{defbQ8}
 \bQ(\xi, \eta) =  \phi_{8, \xi, \eta}(t_0)
\end{equation}
Let $j>0$, $ |\tau | \leq C$,  
$(u, v)\in {\rm supp}{\wh\Phi}\times {\rm supp}{\wh\Phi}$. 
Suppose that
$t_0(u, v), t_0(u-\tau, v+ 2^{-j}\tau)\in [1/16, 9/16]$
exist. 
And let $\ti\bQ_\tau$
be defined by
$$
 \ti\bQ_{\tau}(u, v) = \bQ(u, v) -\bQ(u-\tau, v+ 2^{-j}\tau)\,.
$$
If $j$ is large enough (larger than a constant), 
then the determinant of the Hessian matrix 
of $\ti\bQ_\tau$ satisfies
\begin{equation}\label{8bigHessian2} 
 \big|{\rm det} H(\ti\bQ_\tau)\big|\geq C\tau^2\,,
\end{equation}
where $H(\ti\bQ_{\tau})$ denotes the Hessian matrix. 
\end{lemma}

We omit the proof of this lemma since it is similar to the proof 
of Lemma \ref{4Hessian2}. Finally we need to verify that 
$\ti\bQ_\tau$ is not degenerate.

\begin{lemma}\label{8nondeg}
Let $\tilde{\bQ}_\tau$ be the function defined as in Lemma \ref{8Hessian2}.
If $j$ is large enough, then
\begin{equation}\label{nondeg8}
 \bigg|\frac{\partial^2 \tilde{\bQ}_\tau}{\partial u\partial v}(u, v)\bigg| \geq C_\beta\tau\,
\end{equation}
\begin{equation}\label{nondeg81}
 \bigg|\frac{\partial^3 \tilde{\bQ}_\tau}{\partial^2 u\partial v}(u, v)\bigg| \geq C_\beta 2^{-m}\tau\,
\end{equation}
hold for all $(u, v)\in {\rm supp}\wh\Phi\times {\rm supp}\wh\Phi$.
\end{lemma}
\begin{proof}
Clearly 
$$
\frac{\partial^2 \tilde{\bQ}_\tau}{\partial u\partial v}(u, v)
=\frac{\partial^2 {\bQ}}{\partial u\partial v}(u, v)
-\frac{\partial^2 {\bQ}}{\partial u\partial v}(u-\tau, v+\alpha 2^{-j}\tau)\,.
$$
We can estimate $|\frac{\partial^2 \tilde{\bQ}_\tau}{\partial u\partial v} |$ by
$$
\bigg|\frac{\partial^2 {\bQ}}{\partial u\partial v}(u, v) -
\frac{\partial^2 {\bQ}}{\partial u\partial v}(u-\tau, v)\bigg|
+\bigg|\frac{\partial^2 {\bQ}}{\partial u\partial v}(u-\tau, v)-
\frac{\partial^2 {\bQ}}{\partial u\partial v}(u-\tau, v+\alpha 2^{-j}\tau)\bigg|\,.
$$
By the mean value theorem, the first term in the previous sum is majorized by
$$
\bigg|\frac{\partial^3 {\bQ}}{\partial^2 u\partial v}(u-\theta\tau, v)\bigg|\tau
$$
for some $\theta\in [0,1]$.
From the proof of Lemma \ref{lemHessian0}, we have 
\begin{equation}\label{8dudv}
\frac{\partial^2 {\bQ}}{\partial u\partial v}(u,v)=
\frac{- 2 t_0(u, v)}{\phi''_{8,u,v}(t_0)}\,.
\end{equation} 
Thus 
\begin{equation}\label{8d2udv}
\frac{\partial^2 {\bQ}}{\partial^2 u\partial v}(u,v)=
\frac{- 2
 \frac{\partial t_0}{\partial u}(u, v) \big(\phi''_{8, u, v}(t_0)
 - \beta(\beta+1)(\beta+2)2^{-m-1}|t_0|^{-\beta-2} \big)
}{\big(\phi''_{8,u,v}(t_0)\big)^2}\,,
\end{equation}
which is equal to 
$$
 \frac{ - 2
 \frac{\partial t_0}{\partial u}(u, v) \big(
 2 v -
 {\ti C}_\beta 2^{-m-1} |t_0|^{-\beta-2} \big)
}{\big(\phi''_{8,u,v}(t_0)\big)^2}\,,
$$
where ${\ti C}_\beta $ is a number such that $|{\ti C}_\beta|\sim \beta(\beta+1)^2$.
Since $m \geq b_\beta$ is a large number and $|v|\geq C$ 
whenever $v\in {\rm supp}\wh\Phi$, we have 
 \begin{equation}\label{8d2udv1}
\bigg|\frac{\partial^2 {\bQ}}{\partial^2 u\partial v}(u,v)\bigg|\geq 
 C_\beta   \,
\end{equation}
for all $(u, v)\in {\rm supp}\wh\Phi\times {\rm supp}\wh\Phi  $.
Thus we obtain 
\begin{equation}\label{e1stterm8}
\bigg|\frac{\partial^2 {\bQ}}{\partial u\partial v}(u, v) -
\frac{\partial^2 {\bQ}}{\partial u\partial v}(u-\tau, v)\bigg|
\geq C_\beta \tau\,,
\end{equation}
for all all $(u, v)\in {\rm supp}\wh\Phi\times {\rm supp}\wh\Phi $.\\

Notice that $\frac{\partial^2 {\bQ}}{\partial u\partial^2 v}(u,v)$ equals to 
$$
 \frac{- 2  \big(
 \frac{\partial t_0}{\partial v}(u, v)\phi''_{8, u, v}(t_0)
 +\frac{\partial t_0}{\partial v}(u, v)\beta(\beta+1)(\beta+2)2^{-m-1}|t_0|^{-\beta-2} 
- 2\big)
}{\big(\phi''_{8,u,v}(t_0)\big)^2}\,,
$$
which is clearly bounded by $C_\beta$.
The mean value theorem then yields 
\begin{equation}\label{e2ndterm8}
\bigg|\frac{\partial^2 {\bQ}}{\partial u\partial v}(u-\tau, v)-
\frac{\partial^2 {\bQ}}{\partial u\partial v}(u-\tau, v+\alpha 2^{-j}\tau)\bigg|
\leq C_\beta 2^{-j}\tau\,.
\end{equation}
From (\ref{e1stterm8}) and (\ref{e2ndterm8}), we have (\ref{nondeg8}) if $j$ is large enough.
(\ref{nondeg81}) can be proved similarly. We omit the details. 

\end{proof}

\subsection{Proof of Lemma \ref{osc8bj}}
We now prove Lemma \ref{osc8bj}.  It is quite similar to the proof
of Lemma {\ref{osc4bj}}. 
Let $b_1=1-2^{-j}$ and $b_2=2^{-j}$. 
Changing variable 
$\xi\mapsto \xi-\eta$ and $\eta\mapsto b_1\xi+b_2\eta$, 
we have that $ \ti\La_{j, \beta, m, 8}(f_1,f_2, f_3) $
equals to
$$
\int\!\!\!\int\!\! f_1(\xi-\eta) f_2(b_1\xi+b_2\eta) f_3(\xi)
\wh\Phi(\xi-\eta)\wh\Phi(b_1\xi+b_2\eta) e^{i2^{\beta j+m+1}
\phi_{8, \xi-\eta, b_1\xi+b_2\eta }(t_0)} d\xi d\eta.
$$
Thus by Cauchy-Schwarz we dominate $|\ti\La_{j, \beta, m, 8}|$
by
$$
 \big\| \bT_{j, \beta, m, 8}(f_1, f_2)\big\|_2\|f_3\|_2\,,
$$
where $\bT_{j, \beta, m, 8}$ is defined by
$$
\bT_{j, \beta, m, 8}(f_1, f_2)(\xi)
= \int\!\! f_1(\xi-\eta) f_2(b_1\xi+b_2\eta) 
\wh\Phi(\xi-\eta)\wh\Phi(b_1\xi+b_2\eta) e^{i2^{\beta j+m+1}
\phi_{8, \xi-\eta, b_1\xi+b_2\eta }(t_0)} d\eta.
$$ 
It is easy to see 
 $\big\| \bT_{j, \beta, m, 8}(f_1, f_2)\big\|_2^2$ equals to
$$
\int\bigg(\int\!\int F(\xi, \eta_1, \eta_2) G(\xi, \eta_1, \eta_2)
e^{ i 2^{\beta j+m+1} \big(\phi_{8,\xi-\eta_1, b_1\xi+b_2\eta_1 }(t_0)-
\phi_{8, \xi-\eta_2, b_1\xi+b_2\eta_2 }(t_0) \big)}
  d\eta_1d\eta_2 \bigg)  d\xi\,, 
$$
where  
$$
F(\xi, \eta_1, \eta_2)= \big(f_1\wh\Phi\big)(\xi-\eta_1 )
\overline{\big(f_1\wh\Phi\big)(\xi-\eta_2)  }
$$
$$
G(\xi, \eta_1,\eta_2)=\big(f_2\wh\Phi\big)(b_1\xi+b_2\eta_1 )
 \overline{\big(f_2\wh\Phi\big)(b_1\xi+b_2\eta_2)}\,.
$$
Changing variables $\eta_1\mapsto \eta$ and $\eta_2\mapsto \eta+\tau$, 
we see that $\big\|\bT_{j, \beta, m, 8}(f_1, f_2) \big\|_2^2$ equals to
$$
\int\!\bigg(\int\!\!\int F_\tau(\xi-\eta) G_{\tau}(b_1\xi+b_2\eta)
 e^{i 2^{\beta j+m+1}\big(\phi_{8,\xi-\eta, b_1\xi+b_2\eta }(t_0)-
\phi_{8, \xi-\eta-\tau, b_1\xi+b_2(\eta+\tau) }(t_0)
\big)}  d\xi d\eta                             \bigg)d\tau\,, 
$$
where 
$$
F_\tau(\cdot) = \big(f_1\wh\Phi\big)(\cdot )
\overline{\big(f_1\wh\Phi\big)(\cdot -\tau )  }
$$
$$
G_\tau(\cdot) = \big(f_2\wh\Phi\big)(\cdot )
\overline{\big(f_2\wh\Phi\big)(\cdot + b_2\tau )  }\,.
$$
Changing coordinates to $(u,v)=(\xi-\eta, b_1\xi+b_2\eta)$, 
the inner integral becomes
$$
\int\!\int F_\tau(u)G_\tau(v) e^{i2^{\beta j+m+1}
 \ti\bQ_{\tau,m} (u, v)  } du dv\,,
$$
where $\ti\bQ_{\tau,m}$ is defined by 
$$
\ti\bQ_{\tau,m}(u, v)= \bQ(u, v)-\bQ(u-\tau, v+b_2\tau)\,,
$$
and here $\bQ_{u, v}=\phi_{8, u, v}(t_0(u, v))$.

By Lemma \ref{nondeg8} and a theorem of Phong and Stein \cite{PSt1},
we dominate $\big\| \bT_{j, \beta, m, 8}(f_1, f_2)\big\|_2^2$  by 
$$
C_\e
\int_{-10}^{10} \min\big\{1, 2^{-(\beta j+m+1)/2}\tau^{-1/2} \big\} 
\big\|F_\tau\big\|_2 \big\|G_\tau \big\|_2   d\tau\,.
$$
By Cauchy-Schwarz inequality it is bounded by
$$
 C_\e 2^{-(\beta j+m)(1-\e)/2} 2^{j/2}\|f_1\|_4^2\|f_2\|_2^2\,,
$$
for any $\e>0$.  Thus we have 
\begin{equation}\label{8est422}
 \big|\ti \La_{j, \beta, m, 8}(f_1, f_2, f_3)\big|\leq 
C_\e 2^{-(\beta j+m)(1-\e)/4} 2^{j/4}\|f_1\|_4\|f_2\|_2\|f_3\|_2\,.
\end{equation}
Choose $\e_0$ to be $(\beta-1)/5$. We get
\begin{equation}\label{8est422final}
 \big|\ti \La_{j, \beta, m, 8}(f_1, f_2, f_3)\big|\leq 
C_\beta 2^{-\e_0(\beta j+m)} \|f_1\|_4\|f_2\|_2\|f_3\|_2\,.
\end{equation}

We thus complete the proof of Lemma {\ref{osc8bj}},
and therefore the proof for the case $T_{j,\beta, 8}$.

\end{document}